\documentclass[12pt]{article} 
\usepackage[letterpaper,left=2.5cm,top=2.5cm,right=2.5cm,bottom=2.5cm]{geometry}  
\usepackage{latexsym} 
\usepackage{layout}
\usepackage{float}
\usepackage{framed}
\usepackage{textcomp}
\usepackage{amstext}
\usepackage{multicol} 
\usepackage{longtable} 
\usepackage{lscape} 
\usepackage{multirow} 
\usepackage{array} 
\usepackage{amssymb}
\usepackage{amsmath} 
\usepackage{amsthm}
\usepackage[T1]{fontenc} 
\usepackage{graphicx,tikz}
\usepackage{verbatim}

\newtheorem{theorem}{Theorem}[section]
\newtheorem{lemma}[theorem]{Lemma}

\theoremstyle{definition}
\newtheorem{definition}[theorem]{Definition}

\newtheorem{corollary}[theorem]{Corollary}

\theoremstyle{remark}

\newtheorem*{remark*}{Remark}

\theoremstyle{plain}
\newtheorem{prop}[theorem]{Proposition}

\usepackage[pdftex,colorlinks=true,linkcolor=blue,citecolor=black,urlcolor=cyan]{hyperref}

\newlength{\LL}

\makeatletter
\newcommand{\@authorinfos}{}

\newcommand{\authorinfo}[3]{
  \g@addto@macro{\@authorinfos}{%
    \textit{#1}
    \par\noindent #2;
    \par\noindent \textit{Email address:} \href{mailto:#3}{#3}
    \par\bigskip \noindent
  }
}

\newcommand{\printauthorinfos}{
  \par\bigskip
  \par\noindent\@authorinfos
}
\makeatother

\renewcommand{\and}{, \, }

\begin{document}
\title{ \Huge \bf Regularity for convex viscosity solutions of sigma-2 Equation \footnote{This work was supported by NSFC (No. 12141103)}}
\author{Chen Ruosi \and Jian Huaiyu \and Tu Xushan \and Zhou Xingchen }

\maketitle

\authorinfo{Chen Ruosi}{Department of Mathematical Sciences, Tsinghua University, Beijing, China}{crs22@mails.tsinghua.edu.cn}
\authorinfo{Jian Huaiyu}{Department of Mathematical Sciences, Tsinghua University, Beijing, China}{hjian@mail.tsinghua.edu.cn}
\authorinfo{Tu Xushan}{Department of Mathematics,  The Hong Kong University of Science and Technology, Hong Kong, China}{maxstu@ust.hk}
\authorinfo{Zhou Xingchen}{School of Mathematics and Statistics, Hainan University, Haikou, 570228, PR China}{zxc3zxc4zxc5@stu.xjtu.edu.cn}

\vspace{-1em}
\hrulefill

\begin{abstract}
We prove an interior $C^{2}$ regularity result for convex viscosity solutions of the quadratic Hessian equation $\sigma_2(D^2u) = f(x)$ in all dimensions, under the assumptions $f\in C^{0,1}$ and $\inf f>0$.
\end{abstract}

\section{Introduction}

Hessian equation is a fundamental model in fully nonlinear partial differential equations. The  existence of weak solutions, including viscosity solution, measure solution and their relation, has been extensively studied; see \cite{Ishii_LionsViscositysolu,
trudinger1990dirichlet,Tru97,TW99,Urb90,Wan09} for example.
In this paper, we study interior regularity  for convex viscosity solutions of the quadratic Hessian equation
\begin{equation}
    \label{eq:sigma_lambda}
    \sigma_{2}(D^{2}u)=\sum_{1\leq i<j\leq n}\lambda_{i}\lambda_{j}=f(x),
\end{equation}
where \(\lambda=(\lambda_{1},\ldots,\lambda_{n})\) are the eigenvalues of the Hessian matrix \(D^{2}u\), and \(f\) is a Lipschitz function with nondegenerate condition \(\inf f>0\).

Equation (1) is closely related  to  the problem  for prescribing scalar curvature  of hypersurfaces,
\begin{equation}
    \sigma_{2}(\kappa)= \sum_{1\leq i<j\leq n}\kappa_{i}\kappa_{j}=f(x),
    \label{eq:sigma_kappa}
\end{equation}
where $\kappa = (\kappa_1,\cdots,\kappa_n)$ denotes the principal curvatures.
This problem has been widely studied, dating back to foundational work on the Minkowski and Weyl problems by Heinz \cite{heinz1959elliptic}, Nirenberg \cite{MR0058265}, Aleksandrov \cite{MR0086338} and Pogorelov \cite{pogorelov1978multidimensional}. Later contributions include those by Cheng-Yau \cite{MR4519019}, Guan-Guan \cite{MR1933079}, Guan-Li-Li \cite{MR2954620}, Guan-Ren-Wang \cite{guan2015global} and many others.

\medskip

A key step to study the regularity for the solutions  is to establish Hessian estimates.
Interior Hessian bounds for strictly $k$-convex solutions were first derived by Chou-Wang \cite{MR1835381} and Sheng-Urbas-Wang \cite{Urbas-Wang04} respectively, using Pogorelov's pointwise technique.
Urbas \cite{Urbas00,Urbas01} obtained $C^2$ estimates in terms of integrals of the Hessian for general $k$-Hessian equations.
 In the case $f\equiv 1$ and $n=3$, Warren-Yuan \cite{warren2009hessian} used a special Lagrangian integral approach and so proved Hessian estimates for \eqref{eq:sigma_lambda}, later Qiu-Zhou \cite{QiuZhouSpecialLagrangian24} extended the result to equation \eqref{eq:sigma_kappa}.
Also in dimension three, Qiu \cite{qiu2019interior,qiu2017interiorHessian} established interior curvature estimates for  equation $\sigma_2(\kappa) = f(X,\nu)$ and interior Hessian estimates for $\sigma_2(D^2 u) = f(x,u,Du)$ respectively. These estimates require the   $C^2$ regularity of $f$ and depend on $\|f\|_{C^2}$.
Under the convexity assumption $\sigma_3>-C$, Guan-Qiu \cite{guan2019interior} obtained $C^2$ estimates for both \eqref{eq:sigma_lambda} and \eqref{eq:sigma_kappa}, which also depend on $\|f\|_{C^2}$.
 More Hessian estimates for the special case $f\equiv 1$ can be found in the work of McGonagle-Song-Yuan \cite{mcgonagle2019hessian} for almost convex solutions via a compactness argument, and Shankar-Yuan \cite{shankar2020hessian} for semi-convex solutions via an integral Jacobi inequality and Legendre-Lewy transform.
More recently, Shankar-Yuan \cite{shankar2025hessian} discovered  that a suitable Jacobi inequality holds only in dimension $n\le 4$, and established Hessian estimates in the critical four-dimensional case. Their results also yield estimates in higher dimensions $n\geq 5$, under a dynamic semi-convexity condition $\frac{\lambda_{\min}}{\Delta u}\ge -c(n)$.

\medskip

A common requirement for the results mentioned above is that the right-hand side $f$ is at least $C^2$ and the Hessian estimates
depend on $\|f\|_{C^2}$.  An exception is the work of Zhou \cite{zhou2025hessian}, which only needs $\|f\|_{C^{0,1}}$ to establish interior $C^2$ estimates in dimension $3$.
In our recent work \cite{chenjianzhou2024PSCE}, we obtained interior Hessian estimates for smooth convex solutions in terms of $\|f\|_{C^{0,1}}$ in all dimensions $n\geq 2$.




When $n=2$, equation \eqref{eq:sigma_lambda} reduces to a Monge-Amp\`ere equation.
 The interior $C^2$-regularity for convex viscosity solutions of  Monge-Amp\`ere equations in all dimensions is known under  the optimal conditions: Caffarelli \cite{c} obtained $C^{2,\alpha}$ regularity for $f \in C^\alpha$ with $0<\alpha<1$, and Jian-Wang \cite{jw} extended it to  the case $f \in C^\alpha$ with $0\leq \alpha\leq1$.
Incidentally,  the corresponding global  $C^{2,\alpha}$ regularity has been obtained  for, respectively,
Dirichlet boundary  problem in \cite{s,tw}, natural boundary problem in  \cite{clw, jz} and oblique boundary  problem in  \cite{jt}.

Based on the above observation and the interior Hessian estimates for equation (1), it is natural to ask the following question:

{\sl Let $u$ be a convex viscosity solution of $\sigma_2(D^2 u) = f$ in $B_1$ with $f\in C^k(B_1)$. What is  the smallest $k\geq 0$ such that $u$ is a classical $C^2$ solution with an interior Hessian estimate
    \begin{equation*}
       |D^2u(0)|\le C(n, \|u\|_{L^\infty(B_1)}, \|f^{-1}\|_{L^\infty(B_1)}, \|f\|_{C^k(B_1)})?
    \end{equation*} }
\medskip

Establishing such regularity needs to overcome two main challenges.
First, one cannot directly approximate a convex viscosity solution by smooth convex solutions. The standard solvability of Dirichlet problem in Caffarelli-Nirenberg-Spruck \cite{MR0806416} yields only smooth 2-convex solutions, to which the existing estimates for convex solutions do not apply.
Second, for these smooth 2-convex approximating functions, the Jacobi inequality, which is a key tool for Hessian estimates, may fail, as mentioned in \cite{shankar2025hessian}.

Nevertheless, a few regularity results are available. Mooney \cite{MR4246798} proved the strict 2-convexity   of convex viscosity solution to $\sigma_2(D^2 u) \geq 1$  and the interior $C^2$ regularity, hence  smoothness, of convex viscosity solution to $\sigma_2(D^2 u) = 1$. In dimension four, Shankar and Yuan \cite{shankar2025hessian} removed the convexity assumption and obtained  the  interior $C^2$ regularity  of  viscosity solution to $\sigma_2(D^2 u) = 1$.  Very recently, Fan \cite{Fan} extended  this result to the equation $\sigma_2(D^2 u) = f(x)$ in dimension four, and proved $u\in C^{3, \alpha} $ for any $\alpha\in (0, 1)$   if $f\in C^{1,1}$ and $f>0$,
where $u$ is a continuous viscosity  solution  to $\sigma_2(D^2 u) = f(x)$  satisfying  $\Delta u>0$ in the viscosity sense.
\medskip

In this paper, we try   to answer   the above question for  as small $k$ as possible.

\begin{theorem}\label{Thm_sigma}
    Let $u$ be a convex viscosity solution to equation \eqref{eq:sigma_lambda} in $B_{2} $ with $f \in C^{0,1}(B_2)$ and $\inf_{B_{2}}f>0$.
    Then $u \in C^{2}(B_2)$, and
    $$ \|u\|_{C^{2}(\overline{B_{1/4}})}\le C(n,\|u\|_{L^\infty(B_{2})}, \|f^{-1}\|_{L^\infty(B_{2})}, \|f\|_{C^{0,1}(B_{2})}). $$
\end{theorem}

\bigskip

Although we are not sure that $k=1$ is the smallest for  the above question, we do know the smallest $k$ must be positive, since for  a  merely continuous function $f$, there exist convex viscosity solutions that fail to be $C^{1,1}$.
This follows by adapting Wang's counterexample \cite[Theorem 1]{Wangcounterexample} for the Monge-Amp\`ere equation in dimension two to higher dimensions $n\geq 3$. Consider the function
\begin{equation*}
    u(x,y,x_3,\cdots,x_n) = U(x,y) = \frac{x^2}{\ln|\ln (x^2+y^2)|}+y^2\ln|\ln (x^2+y^2)|.
\end{equation*}
Since $U$ is a viscosity solution of $$\det D^2 U = f = 4 + O\!\left( \dfrac{\ln|\ln (x^2+y^2)|}{\ln (x^2+y^2)}\right)\;  \text{in} \; \mathbb{R}^2, $$
 by Theorem \ref{thm:appendix-viscosity} in the Appendix, we see that  $u$ is a viscosity solution of
 $$\sigma_2 (D^2u) = 4 + O\!\left( \dfrac{\ln|\ln (x^2+y^2)|}{\ln (x^2+y^2)}\right)\in C^0(B_2) \;  \text{in} \; \mathbb{R}^n .$$
  But $ u\notin C^{1,1}$ because $D^2u$ is unbounded near the origin.
\medskip

Theorem \ref{Thm_sigma} is proved by approximation.
We construct a sequence of smooth 2-convex solutions $\{v_k\}$ by the solvability of the Dirichlet problem \cite{MR0806416}. A stability estimate (Proposition \ref{prop:stability}) ensures that $v_k$ converges uniformly to the original viscosity solution $u$.

\medskip

The key step is to derive a uniform $C^{1,1}$ bound for $v_k$.
We begin by proving a boundary Jacobi inequality for smooth 2-convex solutions.
This inequality remains valid even in regions where the classical Jacobi inequality fails, because in those regions the linearized operator has a better elliptic structure after a conformal normalization.
However, integration by parts over different regions would typically require handling integrals on their common interface.
The advantage of our boundary Jacobi approach is that it introduces a positive gradient term, which links the integrals on different regions and thereby avoids separate boundary integrals.
This contrasts with the classical method, where the corresponding gradient term in the Jacobi inequality must be strictly negative.

\medskip

To apply the boundary Jacobi inequality, we have to construct suitable cutoff functions.
We invoke Mooney's strict 2-convexity result \cite{MR4246798} to obtain a 2-convex truncation function $w$ that lies uniformly above $u$, and hence above $v_k$ for large $k$, near the origin.
Then a suitable choice of cutoff function is $[(w-v_k)^+]^6$, which leads to a uniform $C^{1,1}$ bound of $v_k$ through a boundary Jacobi approach.
The Evans-Krylov-Safonov theory then provides $C^{2,\alpha}$ estimates, and a compactness argument implies that $v_k$  converges to the original viscosity solution $u$
in the $C^2$ space.

\medskip

The paper is organized as follows. In Section 2, we establish the boundary Jacobi inequality and then derive Pogorelov-type estimates for smooth 2-convex solutions. In Section 3, we construct approximating solutions, prove their uniform convergence, and apply the estimates from Section 2 to establish the uniform $C^{1,1}$ bound, thereby proving Theorem \ref{Thm_sigma}. The Appendix collects some results used in the proof, including a quadratic optimization lemma and the relation between viscosity solution of $k$-Hessian equation and that of Monge-Amp\`ere equation.

\medskip

Throughout this paper, $B_r = B_r(0)$ denotes the open ball of radius $r$ centered at the origin, and $C(\cdots)$ denotes a positive constant depending only on the listed quantities, which may change
from  line to line.

\section{ Jacobi inequality approach for 2-convex solutions}
\label{section:Jacobi}

A function $u$ is called 2-convex if $\Delta u > 0$ and $\sigma_{2} (D^2u)>0$.  In this section
we consider smooth solutions to quadratic Hessian equation in general dimensions $n \geq  2$,
\begin{equation*}
    F(D^2u)=\sigma_{2}(\lambda) = \sum_{1\leq i < j \leq n} \lambda_i \lambda_j = f(x),
\end{equation*}
where $\lambda_i$'s are the eigenvalues of the Hessian matrix $D^2 u$. We also assume that $f(x)$ is  smooth, and   $f \geq f_0$  for a positive constant $f_0>0$.   We work in the positive branch $\Delta u > 0$, which means that the considered solutions $u$ are smooth and   2-convex.    Hence we have
\begin{equation}
    \label{eq:Delta-u-lower}
    (\Delta u)^2 = \sum^n_{i=1} \lambda_i^2 + 2\sigma_2 >  2f_0.
\end{equation}

We denote for $v \in C^2$,
\[ \Delta_{F}v = \sum_{i,j=1}^{n} F_{ij} \partial_{ij}v = \sum_{i,j=1}^{n} \partial_{j}(F_{ij}\partial_{i}v),\qquad  |\nabla_{F}v|^{2}=\sum_{i,j=1}^{n}F_{ij}D_{i}vD_{j}v,\]
where $F_{ij} = \frac{\partial F}{\partial u_{ij}} = \Delta u \,\delta_{ij} - u_{ij}$.

Assuming that $\lambda_1 \geq \cdots \geq \lambda_n$ and $D^2 u$ diagonal at a point, we have $F_{ii} = \Delta u - \lambda_i$ at the point. The following relations from Shankar-Yuan \cite{shankar2025hessian} describe the degeneracy of the linearized operator $\Delta_F$:
\begin{equation}
\label{eq:f_bounds}
\begin{split}
\frac{f_0}{\Delta u}&\leq F_{11}\leq \left(\frac{n-1}{n}\right)\Delta u,\\
\left(1-\frac{1}{\sqrt{2}}\right)\Delta u&\leq F_{ii}\leq 2\left(\frac{n-1}{n}\right)\Delta u,\qquad i\geq 2.
\end{split}
\end{equation}
 The  inequality
\begin{equation}
\label{eq:sigma1_bound}
(n-2)\Delta u>-n\lambda_{n}
\end{equation}
 will be useful.

\subsection{Jacobi Inequality  }

In our previous work \cite{chenjianzhou2024PSCE}, we developed a special integral approach called boundary Jacobi inequality approach that works for convex smooth solutions. Now we want to extend it to general situations.
 We first prove a dynamic Jacobi inequality  with general right hand side function $f(x)$.  In the special case $f\equiv 1$,  such a Jacobi inequality was obtained in Shankar-Yuan \cite[Proposition 2.1]{shankar2025hessian}.  Our results match the dynamic condition in \cite{shankar2025hessian}.

\begin{lemma}
\label{lemma:from_1_to_f}
Let \(u\) be a smooth 2-convex solution to equation \eqref{eq:sigma_lambda} with \(f\in C^{2}, f>0\). Then
\[\Delta_{F}\Delta u-\frac{4}{3}\frac{|\nabla_{F}\Delta u|^{2}}{\Delta u} \geq \Delta f-\frac32\frac{|Df|^{2}}{f} \]
under the condition
\begin{equation}
    \frac{\lambda_n}{\Delta u} + \frac{1}{4} \geq 0.
    \label{eq:semi-convexity}
\end{equation}
\end{lemma}

\begin{proof}  For any given $x\in B_2$,
    after a rotation transform at \(p=x\), we may assume that \(D^{2}u(p)\) is diagonal.  It is sufficient to prove the desired inequality at $p$.
    By the assumption, we have
    $$(F_{ij})_{n\times n}={\rm diag}(F_{ii}) = {\rm diag}(\Delta u - \lambda_i),$$ which implies
      the following formulas at \(x=p\):
    \begin{align}
    \label{eq:DeltaF_D}
    \Delta_{F}\Delta u & = \sum_{i,j=1}^n F_{ij} \partial_{ij}\Delta u = \sum_{i=1}^{n} F_{ii} \partial_{ii}\Delta u , \\
    \label{eq:nablaF_D}
    |\nabla_{F} \Delta u|^{2} & = \sum_{i,j=1}^n F_{ij} (\Delta u_i) (\Delta u_j) =  \sum_{i=1}^{n} F_{ii} (\Delta u_{i})^{2}.
    \end{align}
    We want to replace the fourth order derivatives \(\partial_{ii}\Delta u=\sum_{k=1}^{n}\partial_{ii}u_{kk}\) in \eqref{eq:DeltaF_D} by third derivatives.
    For this purpose,
    differentiate equation \eqref{eq:sigma_lambda} with respect to $x_k$, we have
    \begin{equation}
        \label{eq:diff-main-once}
        \sum^n_{i,j=1}F_{ij}u_{ijk} = \sum_{i=1}^n F_{ii} u_{iik} = f_k
    \end{equation}
    for each $1\le k\le n$. Furthermore,
    differentiate equation \eqref{eq:diff-main-once} with respect to $x_k$, we have
    $$ \sum^n_{i,j=1}\partial_k F_{ij} \partial_{ij}u_k+\sum^n_{i,j=1}F_{ij}\partial_{ij}u_{kk}=f_{kk}. $$
    Summing up the equality for $k$ from 1 to $n$,  we get
    \begin{align*}
        \sum^n_{i,j=1}F_{ij}\partial_{ij}\Delta u & = \Delta f -\sum^n_{i,j,k=1}\partial_k F_{ij}\partial_{ij}u_k
        = \Delta f - \sum^n_{i,j,k=1}(\Delta u_k\delta_{ij}-u_{ijk})u_{ijk} \\
        & = \Delta f +\sum^n_{i,j,k=1}u^2_{ijk}-\sum^n_{i=1}(\Delta u_i)^2\geq \Delta f + 3\sum^n_{j\neq i}u^2_{jji}+\sum^n_{i=1}u^2_{iii}-\sum^n_{i=1}(\Delta u_i)^2.
    \end{align*}
    So we obtain
    \begin{equation}
    \label{eq:DeltaF_convex}
    \begin{aligned}
    \Delta_F\Delta u -\delta \frac{|\nabla_F \Delta u|^2}{\Delta u}
    & \ge \Delta f + 3\sum^n_{\substack{i,j=1 \\ i\neq j}}u^2_{jji}+\sum^n_{i=1}u^2_{iii}-\sum^n_{i=1} (1+\delta\frac{F_{ii}}{\Delta u})(\Delta u_i)^2 \\
    & = \Delta f + \sum_{i=1}^n \left( 3\sum^n_{\substack{j=1\\ j\neq i}}u^2_{jji} + u^2_{iii}- \left(\sum_{j=1}^n \sqrt{k_i} u_{jji}\right)^2\right),
    \end{aligned}
    \end{equation}
    where $k_i = 1 + \delta \frac{F_{ii}}{\Delta u}$ at $x=p$, and $\delta >1$ and will be determined later.
    We fix $i$ and denote  \(t=t_{i}=(u_{11i},\ldots,u_{nni})^T\) at $x=p$ and $k=k_i$. Consider an optimization problem:
    \[ Q_{min}:=
    \min \{ Q(y): y \in \mathbb{R}^n \   \text{s.t.}  \quad d^T y = c\}, \ \ \  \  Q(y) = y^T A y - (b^T y)^2,
    \]
    where $A = \operatorname{diag(3,\cdots,\overset{(i)}1,\cdots,3)}, b = \sqrt{k}(1,\cdots,1)^T, d = (F_{11}, \cdots, F_{nn})^T$ and $c = f_i$, all of which take values at $x=p$.
    Our goal is to prove \(Q(t)\geq -C(n)\frac{|Df(p)|^2 }{f(p)}\), which will follow from \(Q_{\min}\geq -C(n)|Df(p)|^2/f(p) \).

    By Lemma \ref{lemma:optim}, we choose
    \begin{align*}
        \alpha &= b^TA^{-1}b = \frac{(n+2)k}{3}, \\
        \beta &= b^T A^{-1} d = \sqrt{k} \cdot \frac{(n+1)\Delta u - 2 \lambda_i}{3}, \\
        \gamma &= d^T A^{-1} d = \frac{n+1}{3}(\Delta u)^2 - \frac{4}{3} \Delta u \lambda_i + \frac{2}{3} \lambda_i^2 - \frac{2}{3} f.
    \end{align*}
    From Lemma \ref{lemma:optim}, we  see that  $Q$ is  bounded from below  if $\beta^2 - \gamma(\alpha-1)>0$,  which is equivalent to
    $$ k \left(\frac{(n+1)\Delta u - 2 \lambda_i}{3} \right)^2 - \left[\frac{n+1}{3}(\Delta u)^2 - \frac{4}{3} \Delta u \lambda_i + \frac{2}{3} \lambda_i^2 - \frac{2}{3} f\right] \cdot \left( \frac{(n+2)k}{3}-1 \right)>0 $$
    at $x=p$. Since $f>0$ and $k>1$, it suffices to make sure
    \begin{equation}
    \label{eq:k_criterion_no_f}
         k \left(\frac{(n+1)\Delta u - 2 \lambda_i}{3} \right)^2- \left[\frac{n+1}{3}(\Delta u)^2 - \frac{4}{3} \Delta u \lambda_i + \frac{2}{3} \lambda_i^2 \right] \cdot \left( \frac{(n+2)k}{3}-1 \right)>0.
    \end{equation}
    Let $y = \frac{F_{ii}}{\Delta u}>0$, then $y>0$. Substituting $k = 1+\delta y$ and $\lambda_i = (1-y) \Delta u$ in   the above inequality,    it  is equivalent to
    \begin{equation}
        q_\delta(y) = -2n\delta y^2 + \big[6-2n + 4(n-1)\delta \big]y + (n-1)(4-3\delta) >0.
        \label{eq:q_delta(y)}
    \end{equation}
    Choose $\delta = \frac{4}{3}$, then $ q_\delta(y)>0$ when $y\in (0,\frac{5n+1}{4n})$, which  can be guaranteed by the condition (6), since
    \[
    0<y=\frac{F_{ii}}{\Delta u}\leq \frac{\Delta u-\lambda_n}{\Delta u}\le \frac{5}{4}< \frac{5n+1}{4n}.
    \]

    Now we compute the lower bound of $Q$. From Lemma \ref{lemma:optim}, since we have $1-\alpha + \frac{\beta^2}{\gamma}>0$, the minimum is $ Q_{\min} = \frac{c^2(1-\alpha)}{\gamma(1-\alpha) + \beta^2}$. Notice that
    \begin{eqnarray*}
        \gamma(1-\alpha)+\beta^2
       & = & \frac{(\Delta u)^2y}{9} \cdot q_{\frac{4}{3}}(y)+\frac{2}{3}\left( \frac{(n+2)k}{3}-1 \right){f(p)}
        \geq \frac{2}{3}(\alpha-1)f(p)>0,
    \end{eqnarray*}
    where the inequality comes from \eqref{eq:k_criterion_no_f}.
    Then
    $$ Q_{\min} = \frac{c^2(1-\alpha)}{\gamma(1-\alpha) + \beta^2} \geq \frac{f_i^2(1-\alpha)}{\frac{2}{3}(\alpha-1)f} \geq -\frac{3}{2} \frac{|Df(p)|^2}{f(p)}.   $$
    This completes the proof.
\end{proof}

We next prove a lower bound for $\Delta_F\Delta u $ without the dynamic semi-convexity condition \eqref{eq:semi-convexity}.

\begin{lemma}
 \label{lemma:Jacobi_only_Delta_F}
Let \(u\) be a smooth 2-convex solution to equation \eqref{eq:sigma_lambda} with \(f\in C^{2}, f >0\). Then
\[ \Delta_{F}\Delta u \geq \Delta f-\frac{3}{2}\frac{|Df|^{2}}{f}. \]
\end{lemma}

\begin{proof}  For any given $x\in B_2$, as in the beginning of the proof of Lemma 2.1, we
    assume $D^2u$ is diagonal at $x = p$. Then the same computation as \eqref{eq:DeltaF_convex} leads to
    $$ \Delta_F\Delta u \geq  \Delta f + \sum_{i=1}^n \left( 3\sum^n_{\substack{j=1\\ j\neq i}}u^2_{jji} + u^2_{iii}- \left(\sum_{j=1}^n u_{jji}\right)^2\right). $$
    Fix $1\le i\le n$, consider an optimization problem:
    \[ \widetilde{Q}_{\min}=
    \min\{ \ \widetilde{Q}(t): t \in \mathbb{R}^n  \quad  \text{s.t.}  \quad d^T t = c\},\ \ \  \widetilde{Q}(t)= t^T A t - (b^T t)^2 ,
    \]
    where $A = \operatorname{diag(3,\cdots,\overset{(i)}1,\cdots,3)}, b = (1,\cdots,1)^T, d = (F_{11}, \cdots, F_{nn})^T$ and  $c = f_i$, all of which take values at $x=p$.
    The goal is to show that \( \widetilde{Q} \geq -C(1)|Df(p)|^2/f(p) \).
    By Lemma \ref{lemma:optim}, we have
    \begin{align*}
        \alpha &= b^TA^{-1}b = \frac{n+2}{3}, \\
        \beta &= b^T A^{-1} d = \frac{(n+1)\Delta u - 2 \lambda_i}{3}, \\
        \gamma &= d^T A^{-1} d = \frac{n+1}{3}(\Delta u)^2 - \frac{4}{3} \Delta u \lambda_i + \frac{2}{3} \lambda_i^2 - \frac{2}{3} f.
    \end{align*}
    By Lemma \ref{lemma:optim}, $\widetilde{Q}$ is bounded below provided $\beta^2 > \gamma(\alpha-1)$, which is
    $$  \left(\frac{(n+1)\Delta u - 2 \lambda_i}{3} \right)^2 - \left[\frac{n+1}{3}(\Delta u)^2 - \frac{4}{3} \Delta u \lambda_i + \frac{2}{3} \lambda_i^2 - \frac{2}{3} f\right] \cdot \frac{n-1}{3}>0. $$
    Since $f>0$, it suffices to make sure
    \begin{equation}
        \label{eq:criterion_no_f}
        \left(\frac{(n+1)\Delta u - 2 \lambda_i}{3} \right)^2 - \left[\frac{n+1}{3}(\Delta u)^2 - \frac{4}{3} \Delta u \lambda_i + \frac{2}{3} \lambda_i^2  \right] \cdot \frac{n-1}{3}>0.
    \end{equation}
    Denote $z = \frac{\lambda_i}{\Delta u}$. By  \eqref{eq:sigma1_bound}, we  see that $z   \in(-\frac{n-2}{n},1)$. Then inequality \eqref{eq:criterion_no_f} is equivalent to
    $$ q(z) := -(n-3) z^2 -4 z + n+1 >0. $$
    As $ q(-\frac{n-2}{n}) = \frac{12(n-1)^2}{n^2} >0$ and $q(1)=0$, we get $q(z)>0$ for all $z\in (-\frac{n-2}{n},1)$.

    From Lemma \ref{lemma:optim}, since $1-\alpha + \frac{\beta^2}{\gamma}>0$, we have $ \widetilde{Q}_{\min} = \frac{c^2(1-\alpha)}{\gamma(1-\alpha) + \beta^2}$. Notice that
    \begin{equation*}
        \gamma(1-\alpha)+\beta^2
        = \frac{2}{9}(\Delta u)^2 \cdot q(z) + \frac{2(n-1)}{9} f \geq  \frac{2(n-1)}{9} f,
    \end{equation*}
    where the last inequality comes from \eqref{eq:criterion_no_f}.
    Since $1-\alpha = \frac{1-n}{3}$, we thus have
    $$ \widetilde{Q}_{\min} \geq -\frac{3}{2f(p)} c^2 \geq -\frac{3}{2f(p)} |Df(p)|^2.   $$
    This completes the proof.
\end{proof}

Now we are ready to prove the key proposition of this section, namely a boundary Jacobi inequality without assuming any additional convexity condition of the solution.

\subsection{Boundary Jacobi Inequality }

Take the function
\begin{eqnarray*}
    b = \Delta u + A + e^{|Du|^2},
\end{eqnarray*}
where $A = A(n,\|u\|_{C^{0,1}(B_1)})$ is a positive constant to be determined later.

\begin{prop}[Boundary Jacobi inequality]
\label{lemma:Boundary_Jacobi}
Let $u$ be a smooth 2-convex solution to equation \eqref{eq:sigma_lambda} in $B_2$
and $\Omega \subset B_1$ be a connected open subset in $B_1$.   Then for any $\phi\in C^2(\Omega), \phi>0 $ in $\Omega$, we have
\begin{eqnarray}
\label{eq:boundary-Jacobi-main}
\Delta_F (\phi^6 b) + \frac{1}{2} \frac{|\nabla_F(\phi^6 b)|^2}{\phi^6 b}
\ge 6 \phi^{5} b \Delta_F \phi + \phi^6 \Delta f - C\Delta u\;\; \text{in }\;\; \Omega,
\end{eqnarray}
where $C = C(n, f_0, \|f\|_{C^{0,1}(B_1)}, \|u\|_{C^{0,1}(B_1)},\|\phi\|_{C^{0,1}(\Omega)})$.
\end{prop}

\begin{proof}
    Denote $\varphi := \phi^6, J:= \Delta_F (\varphi b) + \frac{1}{2} \frac{|\nabla_F(\varphi b)|^2}{\varphi b}$ and
    \begin{eqnarray*}
    \Omega_1&:=&\left\{x\in \Omega \mid  \frac{\lambda_n}{\Delta u} + \frac{1}{4}\ge 0 \right\}, \\
    \Omega_2 &:=& \left\{ x\in \Omega \mid \frac{|\nabla_F \varphi|}{\varphi} \leq \frac{1}{6} \frac{|\nabla_F b|}{b} \right\}.
    \end{eqnarray*}
    We will prove the desired inequality \eqref{eq:boundary-Jacobi-main} in $\Omega_1$, $\Omega_2$ and $\Omega_1^c\cap \Omega_2^c$ respectively.

    \textbf{Step 1:} In $\Omega_1$.
    Direct computation shows that
    \begin{equation}
    \label{eq:nabla_F-varphi-b}
        \begin{aligned}
            \Delta_F (\varphi b) & = \varphi \Delta_F b + 2 \sum_{i,j=1}^nF_{ij} \varphi_i b_j + b\Delta_F \varphi \\
            & \geq \varphi \Delta_F b - \frac{5}{6} \frac{|\nabla_F\varphi|^2}{\varphi} b - \frac{6}{5} \frac{|\nabla_F b|^2}{b}\varphi + b\Delta_F \varphi \\
            & = b \left( \Delta_F \varphi - \frac{5}{6} \frac{|\nabla_F\varphi|^2}{\varphi} \right) + \varphi \left( \Delta_F b - \frac{6}{5} \frac{|\nabla_F b|^2}{b} \right).
        \end{aligned}
    \end{equation}
    Since $\varphi = \phi^6$, we have
    \begin{equation}
    \label{eq:nabla_F-varphi}
        \Delta_F \varphi - \frac{5}{6} \frac{|\nabla_F\varphi|^2}{\varphi} = 6 \phi^{5} \Delta_F \phi.
    \end{equation}
    Young's inequality yields
    \begin{equation*}
        |\nabla_F b|^2 = |\nabla_F(\Delta u + e^{|Du|^2})|^2 \leq \left(1+\frac{1}{9}\right)|\nabla_F \Delta u|^2 +  C(n)  |\nabla_F e^{|Du|^2}|^2,
    \end{equation*}
    and thus
    \begin{equation*}
        \Delta_F b - \frac{6}{5} \frac{|\nabla_F b|^2}{b} \geq \Delta_F \Delta u + \Delta_F e^{|Du|^2} - \frac{4}{3} \frac{|\nabla_F \Delta u|^2}{\Delta u} - C(n) \frac{|\nabla_F e^{|Du|^2}|^2}{A}.
    \end{equation*}
    In $\Omega_1$, we can use   Lemma \ref{lemma:from_1_to_f}. As a result, we have
    \begin{equation*}
        \Delta_F \Delta u - \frac{4}{3} \frac{|\nabla_{F}\Delta u|^{2}}{\Delta u} \geq \Delta f-C(n,f_0,\|f\|_{C^{0,1}}) \quad \text{ in } \Omega_1.
    \end{equation*}
    So we have
    $$
    \Delta_F b - \frac{6}{5} \frac{|\nabla_F b|^2}{b} \geq \Delta f-C(n,f_0,\|f\|_{C^{0,1}(B_2)}) + \Delta_F e^{|Du|^2} - C(n) \frac{|\nabla_F e^{|Du|^2}|^2}{A}.
    $$
    Notice that
    $$ |\nabla_F e^{|Du|^2} |^2 = e^{2|Du|^2} |\nabla_F|Du|^2|^2. $$
     Using \eqref{eq:diff-main-once}, we have
    \begin{eqnarray}
        \Delta_F e^{|Du|^2} &=& 2 e^{|Du|^2} \sum_{i,j,k} F_{ij} u_{ki} u_{kj} + 2 e^{|Du|^2} \sum_{i,j,k} u_k F_{ij} u_{ijk} + e^{|Du|^2} |\nabla_F|Du|^2|^2 \notag \\
        &\geq & -C(\|u\|_{C^{0,1}(B_2)}, \|f\|_{C^{0,1}(B_2)}) + e^{|Du|^2} |\nabla_F|Du|^2|^2. \label{eq:Delta_F-exp}
    \end{eqnarray}
    Therefore, taking $A = A(n,\|u\|_{C^{0,1}(B_1)})$ large enough, we have
    \begin{eqnarray}\label{eq:nabla_F-b}
        \Delta_F b - \frac{6}{5} \frac{|\nabla_F b|^2}{b} \geq \Delta f - C(n,f_0,\|f\|_{C^{0,1}(B_2)}, \|u\|_{C^{0,1}(B_2)}).
    \end{eqnarray}
    So, in $\Omega_1$, by \eqref{eq:nabla_F-varphi-b} \eqref{eq:nabla_F-varphi} \eqref{eq:nabla_F-b} we have thus proved that
    \begin{equation*}
        \begin{aligned}
            J &\geq 6 \phi^{5} b \Delta_F \phi + \varphi\big[\Delta f - C(n, f_0, \|f\|_{C^{0,1}(B_2)}, \|u\|_{C^{0,1}(B_2)}) \big] \\
            & \geq 6 \phi^{5} b \Delta_F \phi + \varphi \Delta f - C(n, f_0, \|f\|_{C^{0,1}(B_2)}, \|u\|_{C^{0,1}(B_2)},\|\phi\|_{L^{\infty}(B_2)}).
        \end{aligned}
    \end{equation*}

    \textbf{Step 2:} In $\Omega_2$. We have
    $$ |\sum_{i,j=1}^nF_{ij}\varphi_ib_j| \leq |\nabla_F \varphi| \cdot |\nabla_F b| \leq \frac{\varphi}{6b}  |\nabla_F b|^2 , $$
   where the last inequality is due to the condition of  $\Omega_2$. This implies
    \begin{eqnarray*}
        J &=& \varphi \Delta_F b + 2 \sum_{i,j=1}^nF_{ij} \varphi_i b_j + b\Delta_F \varphi + \frac{1}{2} \left( \frac{b}{\varphi}|\nabla_F \varphi|^2 + 2\sum_{i,j=1}^nF_{ij} \varphi_i b_j + \frac{\varphi}{b}|\nabla_F b|^2  \right) \\
        &\geq & \varphi \Delta_F b + b\Delta_F \varphi + 3 \sum_{i,j=1}^nF_{ij} \varphi_i b_j + \frac{\varphi}{2b} |\nabla_F b|^2 \\
        &\geq & \varphi \Delta_F b + b\Delta_F \varphi.
    \end{eqnarray*}
    From \eqref{eq:nabla_F-varphi}, we know $\Delta_F \varphi \geq 6\phi^{5} \Delta_F \phi$.  On the other hand,
    combining Lemma \ref{lemma:Jacobi_only_Delta_F} and \eqref{eq:Delta_F-exp}, we know
    $$ \Delta_F b = \Delta_F \Delta u + \Delta_F e^{|Du|^2} \geq \Delta f - C(n, f_0, \|u\|_{C^{0,1}(B_2)}, \|f\|_{C^{0,1}(B_2)}) $$
    So, in $\Omega_2$, we have proved
    $$ J \geq 6 \phi^{5} b \Delta_F \phi + \varphi \Delta f - C(n, f_0, \|f\|_{C^{0,1}(B_2)}, \|u\|_{C^{0,1}(B_2)},\|\phi\|_{L^{\infty}(B_2)}) . $$

    \textbf{Step 3:} In $\Omega_1^c \cap \Omega_2^c $.
    Using \eqref{eq:nabla_F-varphi-b} and \eqref{eq:nabla_F-varphi}, we have
    $$ J \geq 6 \phi^{5} b \Delta_F \phi +\varphi \left( \Delta_F \Delta u + \Delta_F e^{|Du|^2} - \frac{6}{5} \frac{|\nabla_F b|^2}{b} \right). $$
    From Lemma \ref{lemma:Jacobi_only_Delta_F}, we have $\Delta_F\Delta u \geq \Delta f - C|Df|^2$.
    To estimate $\Delta_F e^{|Du|^2}$, we may assume $D^2u(x)$ is diagonal at a given point. This, together with (9), yields
    \begin{eqnarray*}
        \Delta_F e^{|Du|^2} &=& 2 e^{|Du|^2} \sum_{i,j,k} F_{ij} u_{ki} u_{kj} + 2 e^{|Du|^2} \sum_{i,j,k} u_k F_{ij} u_{ijk} +  e^{|Du|^2} |\nabla_F|Du|^2|^2 \\
        &\geq& c_1(n) (\Delta u)^3 - C_1(\|u\|_{C^{0,1}(B_2)}, \|f\|_{C^{0,1}(B_2)}),
    \end{eqnarray*}
    where we have used
    \begin{eqnarray*}
        \sum_{i,j,k} F_{ij} u_{ki} u_{kj} = \sum_{i=1}^n (\Delta u -\lambda_i) \lambda_i^2 \geq (\Delta u - \lambda_n)\lambda_n^2 \geq c_1(n) (\Delta u )^3,
    \end{eqnarray*}
    where the last inequality follows from the condition of $\Omega_1^c $. On the other hand,  from the definition of $\Omega_2^c$, we have
    $$ \frac{|\nabla_F b|^2}{b} \leq 36\, b\, \frac{|\nabla_F\varphi|^2}{\varphi^2} \leq C_2(n, f_0,\|\phi\|_{C^{0,1}(B_2)}, \|u\|_{C^{0,1}(B_2)}) (\Delta u)^2 \varphi^{-\frac{1}{3}}, $$
    where we have used  $b \leq C(n, f_0, \|u\|_{C^{0,1}(B_2)}) \Delta u$ in the last inequality. Combining these inequalities, we have
    \begin{eqnarray*}
        J &\geq& 6 \phi^{5} b \Delta_F \phi +\varphi \Delta f - C_1 + c_1(n) (\Delta u)^3 \varphi - C_2 (\Delta u)^2 \varphi^{ \frac{2}{3}} \\
        &\geq& 6 \phi^{5} b \Delta_F \phi +\varphi \Delta f - C(n, f_0,\|f\|_{C^{0,1}(B_2)}, \|u\|_{C^{0,1}(B_2)},\|\phi\|_{C^{0,1}(B_2)}) \Delta u.
    \end{eqnarray*}
    This completes the proof of inequality \eqref{eq:boundary-Jacobi-main} in the whole region.
\end{proof}

\subsection{Pogorelov-type \texorpdfstring{$W^{2,p}$}{} estimates}

\begin{prop}
\label{lemma:Hessian_W2p_estimate}
    Let $u$ be a smooth 2-convex solution of \eqref{eq:sigma_lambda} in $ B_2$ with $f\in C^{\infty}(B_2)$ and $\inf_{B_2} f\ge f_0 >0$.
    Suppose that $\Omega\subset B_1$ is a connected open subset in $B_1$ and   $w \in C^2(\mathbb{R}^n)$ is a 2-convex function satisfying
    $$ w>u  \text{ in } \Omega, \qquad  w=u \text{ on } \partial \Omega. $$
      Then for   $\varphi = (w-u)^6  $  and  every integer $p \geq 1$, we have
    $$ \int_{\Omega} (\Delta u)^p \varphi^{p-1} dx \le  p!\, [C(n,f_0,\|w\|_{C^{0,1}(B_2)}, \|u\|_{L^{\infty}(B_2)},\|f\|_{C^{0,1}(B_2)})]^p. $$
\end{prop}

\begin{proof}
By the gradient estimate by Chou-Wang (Theorem 3.2 in \cite{MR1835381}), we know that $ \|u\|_{C^{0,1}(B_1)}$ can be bounded by a constant depending only on $n, f_0, \|u\|_{L^{\infty}(B_2)}, \|f\|_{C^{0,1}(B_2)}$.
Hence, in the proof below, we may use $\|u\|_{C^{0,1}(B_1)}$, and at the end replace this dependence by the above quantities. We denote constant $C=C(n,f_0,\|w\|_{C^{0,1}(B_2)}, \|u\|_{C^{0,1}(B_1)} ,\|f\|_{C^{0,1}(B_2)})$ which may change from line to line.
    Write $\varphi = \phi^6$, where $ \phi = w-u \in C^2(\Omega)$.
    Then $$\Delta_F \phi = \Delta_F w - \Delta_F u > -2f,$$
    where we use $\Delta_Fw >0 $ since $w$ is 2-convex, and $\Delta_F u = 2f$.  Recalling the function $b$ in the beginning of this subsection,
    applying Proposition \ref{lemma:Boundary_Jacobi}, and observing that the first term in the right hand side of \eqref{eq:boundary-Jacobi-main} can be estimated by
    $$  6 \phi^{5} b \Delta_F \phi \geq -C\Delta u, $$
    we thus have
    $$ \Delta_F(\varphi b) + \frac{1}{2}\frac{|\nabla_F (\varphi b)|^2}{\varphi b} \geq  \varphi\Delta f  - C \Delta u. $$
    For any integer $p\ge 1$,  integrating the above inequality multiplied by the test  function $(\varphi b)^p$, we have
    \begin{equation}
    \label{eq_temp_1}
    (p- \frac{1}{2}) \int_{\Omega}|\nabla_F(\varphi b)|^2[\varphi b]^{p-1} dx \le C\int_{\Omega} b^{p} \varphi^p \Delta u \,dx -\int_{\Omega}  b^p \varphi ^{p+1}\Delta f  \,dx.
    \end{equation}
   To deal with the terms on the right hand side, for any constant $0<\delta<1$ we  integrate by parts and  use Young's inequality to estimate that
    \begin{eqnarray}
         \int_{\Omega} b^p \varphi^p \Delta u \,dx &= &-\int_{\Omega} D[(\varphi b)^p] Du \,dx\notag  \\
         &\leq & \delta p\int_{\Omega} |D(\varphi b)|^2 b^{p-2}\varphi^{p-1} \,dx +\frac p{\delta} \int_{\Omega}|Du|^2 b^{p} \varphi^{p-1} \,dx, \label{eq:recursion_ahead}
    \end{eqnarray}
    and
    \begin{eqnarray}
       - \int_{\Omega}\Delta f \varphi^{p+1} b^{p} dx &=& \int_{\Omega}\varphi Df D(\varphi b)^{p} dx+\int_{\Omega} (\varphi b)^{p}Df D\varphi \,dx \notag \\
        & \leq & \delta p\int_{\Omega} |D(\varphi b)|^2 b^{p-2} \varphi^{p-1} dx
        +\frac{pC}{\delta} \int_{\Omega} b^{p}\varphi ^{p-1}dx .\label{eq:temp_Delta_f}
    \end{eqnarray}
    Since $(\Delta u-\lambda_i)\Delta u\ge f\ge f_0$ for each $1\le i\le n$,  we have
    \begin{equation}
        \label{eq:nabla_F_to_D}
        |\nabla_F(\varphi b)|^2(\varphi b)^{p-1}\ge |\nabla_F(\varphi b)|^2\Delta u \varphi^{p-1} b^{p-2} \geq f_0 |D(\varphi b)|^2 b^{p-2}\varphi^{p-1} .
    \end{equation}
    Substituting inequality \eqref{eq:nabla_F_to_D} into \eqref{eq:recursion_ahead} and \eqref{eq:temp_Delta_f}, and then inserting the resulting inequality into \eqref{eq_temp_1}, we obtain
    \begin{equation*}
        (p- \frac{1}{2}) \int_{\Omega}|\nabla_F(\varphi b)|^2[\varphi b]^{p-1} dx \leq C\delta p \int_{\Omega}|\nabla_F(\varphi b)|^2[\varphi b]^{p-1}\,dx + \frac{pC}{\delta} \int_{\Omega} b^p \varphi^{p-1}\,dx.
    \end{equation*}
    Choosing   a sufficiently small $\delta>0$ and recalling $p\ge 1$, we get
    \begin{equation*}
        \int_{\Omega}|\nabla_F(\varphi b)|^2[\varphi b]^{p-1} dx \leq C \int_{\Omega} b^p \varphi^{p-1} \,dx.
    \end{equation*}
    Since $Du$ is uniformly bounded and $\Delta u\ge \sqrt{2f_0}$ has a universal lower bound, we have $b \leq C \Delta u$.
    Combining this equality with \eqref{eq:recursion_ahead}, we finally obtain    a recursion formula
    \begin{eqnarray*}
        \int_{\Omega} b^{p+1} \varphi^p \,dx \leq C \int_{\Omega} b^p \varphi^p \Delta u \,dx & \leq& \delta p\int_{\Omega} |D[\varphi b]|^2 b^{p-2}\varphi^{p-1} dx +\frac p{\delta} \int_{\Omega}|Du|^2 b^{p} \varphi^{p-1} dx \\
        & \leq & C\delta p \int _{\Omega}|\nabla_F(\varphi b)|^2[\varphi b]^{p-1} \,dx + pC \int_{\Omega} b^p \varphi^{p-1} \,dx \\
        & \leq & p C \int_{\Omega} b^p \varphi^{p-1} \,dx.
    \end{eqnarray*}
    Therefore,
    $$
    \int_{\Omega} b^{p+1}\varphi ^pdx \le pC \int_{\Omega} b^{p}\varphi ^{p-1}dx \leq \cdots \leq p! C^p \int_{\Omega} b^2 \varphi \,dx \leq p! C^p \int_{\Omega} b \,dx \leq p! C^p \int_{\Omega} \Delta u \,dx.
    $$
   Now we choose a cutoff function $\xi \in C_0^{\infty}(B_2),\xi =1 $ in $B_1$, $|D\xi|<10$, then
    \begin{eqnarray*}
        \int_\Omega \Delta u \,dx \leq \int_{B_2} \xi \operatorname{div}(Du) \, dx \leq |\int_{B_2} {D\xi \cdot Du} \,dx |\leq C.
    \end{eqnarray*}
\end{proof}

\subsection{Pogorelov-type \texorpdfstring{$C^{1,1}$}{} estimates}

\begin{prop}
    \label{lemma:Hessian_C11ptype_estimate}
    Let $u$ be a smooth 2-convex solution of \eqref{eq:sigma_lambda} in $ B_2$ with $f\in C^{\infty}(B_2)$ and $\inf_{B_2} f\ge f_0 >0$.
    If $\Omega \subset B_1$ is a connected open subset and $\varphi\in C^2(B_2)$, $\varphi \geq 0$ satisfies $\varphi=0$ on $B_2 \setminus \Omega$,
    then
    $$
        \|\varphi^{2n}\Delta u\|_{L^{\infty} (\Omega)}\le C \int_{\Omega}\varphi^{n} (\Delta u)^{n+1}dx,
    $$
    where $C = C(n,f_0,\|\varphi\|_{C^{0,1}(\Omega)},\|f\|_{C^{0,1}(B_2)})$.
\end{prop}

\begin{proof}
    In this proof, we denote constant $C=C(n,f_0,\|\varphi\|_{C^{0,1}(B_2)}, \|f\|_{C^{0,1}(B_2)})$ which may change from line to line.

    Let $0<\rho<1$, and $\xi\in C^\infty_0(B_{1+\rho})$ be a cutoff function with $\xi=1$ in $B_{1}$, $\xi\ge 0$, $|D\xi|\le 2\rho^{-1} $. Let $p\ge 1,q\ge 2$ be a pair of positive constants satisfying $1\le\frac qp\le 2n$. According to Lemma \ref{lemma:Jacobi_only_Delta_F},  we have
    $$
    \int_{B_2}(\Delta u)^p\varphi^{q}\xi^2\Delta_F\Delta u \, dx \ge\int_{B_2}(\Delta u)^p \varphi^{q}\xi^2 [{\Delta f}-C |Df|^2] dx.
    $$
    Using integration by parts,    $\Delta u-\lambda_i\le 2\Delta u$ and (4), we get
    \begin{align*}
    \int_{B_{1+\rho}}|D\Delta u|^2(\Delta u)^{p-2}\varphi^q\xi^2dx
    & \le C \int_{B_{1+\rho}}(|D \xi|^2+|Df|^2+|D\varphi|^2)(\Delta u)^{p+2}\varphi^{q-2}dx \\
    & \le \frac{C }{\rho^2} \int_{B_{1+\rho}} (\Delta u)^{p+2}\varphi^{q-2} dx.
    \end{align*}
    Then we have
    \begin{align*}
    \int_{B_1}|D[(\Delta u)^{\frac p2}\varphi^{\frac q2}]|^2dx
    \le \frac{C_2p^2q^2}{\rho^2}\int_{B_{1+\rho}}(\Delta u)^{p+2}\varphi^{q-2}dx,
    \end{align*}
    and
    \begin{align*}
    \int_{B_{1}}|(\Delta u)^{\frac p2}\varphi^{\frac q2}|^2 dx
    \le C_1\int_{B_{1+\rho}}(\Delta u)^{p+2}\varphi^{q-2}dx \le \frac{C_2p^2q^2}{\rho^2}\int_{B_{1+\rho}}(\Delta u)^{p+2}\varphi^{q-2}dx .
    \end{align*}
    These two estimates give $(\Delta u)^{\frac p2}\varphi^{\frac q2} \in W^{1,2}(B_{1})$.

    When $n>2$, we have $ W^{1,2}(B_1) \hookrightarrow L^{\frac{2n}{n-2}}(B_1) $ by Sobolev embedding theorem, which, together with the required inequality
    $1\le\frac qp\le 2n$, yields
    \begin{align*}
    \int_{B_{1}}(\Delta u)^{\frac {np}{n-2}}\varphi^{\frac {nq}{n-2}} \, dx
    &\le [\frac{C p^2q^2}{\rho^2}\int_{B_{1+\rho}}(\Delta u)^{p+2}\varphi^{q-2}dx]^{\frac {n}{n-2}} \\
    & \le [\frac{C  p^4}{\rho^2}\int_{B_{1+\rho}}(\Delta u)^{p+2}\varphi^{q-2}dx]^{\frac {n}{n-2}}.
    \end{align*}
     Denote $\gamma = \frac{n}{n-2}$, fix an integer $k_0\ge\frac{\ln n}{\ln\gamma}$ and take $p_0 = \gamma^{k_0}$, $q_0 = 2np_0$ to initiate the iteration. For $k = 1,\cdots,k_0$, set
    \begin{eqnarray*}
        p_{k} &=& \gamma^{-1} p_{k-1}+2=\gamma^{-k}p_0+2\sum^{k-1}_{i=0}\gamma^{-i}, \\
        q_{k} &=& \gamma^{-1} q_{k-1}-2=\gamma^{-k}q_0-2\sum^{k-1}_{i=0}\gamma^{-i}, \\
        r_k &=& 1 + \sum^{k}_{i=1}2^{-(k_0-i+2)}.
    \end{eqnarray*}
    It is easy to verify   that $p_k > n$ and $q_k \ge 2, $  satisfying $\frac{q_k}{p_k} \le 2n$ for all $k$.
    Then the previous inequality means that
    $$
    \int_{B_{r_{k-1}}}(\Delta u)^{p_{k-1}}\varphi^{q_{k-1}}dx
    \le [C \, 2^{2(k_0-k+2)} (p_k-2)^4 \int_{B_{r_{k}}}(\Delta u)^{p_k}\varphi^{q_{k}}\,dx]^{\gamma}.
    $$
    Therefore, an iteration process implies
    \begin{align*}
    \int_{B_{1}}(\Delta u)^{p_0}\varphi^{q_{0}}dx
    &\le C^{\sum^{k_0}_{k=1}\gamma^k}2^{2\sum^{k_0}_{k=1}(k_0-k+2)\gamma^k}\prod ^{k_0}_{k=1}(p_k-2)^{4\gamma^k}[\int_{B_{r_{k_0}}}(\Delta u)^{p_{k_0}}\varphi^{q_{k_0}}dx]^{\gamma^{k_0}}\\
    &\le C^{\sum^{k_0}_{k=1}\gamma^k} 4^{\sum^{k_0}_{k=1}(k_0-k+2)\gamma^k} \prod^{k_0}_{k=1}(n\gamma^{-k}p_0)^{4\gamma^k} [\int_{B_{r_{k_0}}}(\Delta u)^{p_{k_0}}\varphi^{q_{k_0}}dx]^{\gamma^{k_0}} \\
    & = (n^4 C )^{\sum^{k_0}_{k=1}\gamma^k} 4^{\sum^{k_0}_{k=1}(k_0-k+2)\gamma^k} \gamma^{4\sum^{k_0}_{k=1}(k_0-k)\gamma^k} [\int_{B_{r_{k_0}}}(\Delta u)^{p_{k_0}}\varphi^{q_{k_0}}dx]^{p_0}.
    \end{align*}

    Notice that both $\frac{1}{\gamma^{k_0}} \sum^{k_0}_{k=1}\gamma^k $ and $ \frac{1}{\gamma^{k_0}} \sum^{k_0}_{k=1}(k_0-k)\gamma^k $ can   be bounded by some constant $C(n)$. So we get
    \begin{align*}
    \|\varphi^{2n}\Delta u\|_{L^{p_0} (B_{1})} = [\int_{B_{1}}(\Delta u)^{p_0}\varphi^{q_{0}}dx]^\frac 1{p_0} \leq C ^{C(n)} \int_{B_{2}}(\Delta u)^{p_{k_0}}\varphi^{q_{k_0}}dx.
    \end{align*}
    From the choice of $p_0, q_0 $ and $k_0$, we can verify that $p_{k_0} \le n+1 $ and $q_{k_0} \ge n$. So the above estimate implies
    $$ \|\varphi^{2n}\Delta u\|_{L^{p_0} (B_{1})} \leq C \int_{B_{2}}(\Delta u)^{n+1}\varphi^{n}dx. $$
    Let $k_0\rightarrow +\infty$, that is $p_0 \to +\infty$, we obtain the desired   conclusion.

    When $n=2$, take $\gamma>1$ to be any fixed number, for example $\gamma=2$.  Then $W^{1,2}\hookrightarrow L^{\gamma}$. Following the above iteration process, one can still obtain
    $$ \| \varphi^{2n}\Delta u\|_{L^{\infty}(B_{1}) } \leq C \int_{B_{2}}(\Delta u)^{n+1}\varphi^{n}dx .$$
    This completes the proof of the proposition.
\end{proof}

The next Proposition \ref{prop: Hessian} follows immediately from Proposition    \ref{lemma:Hessian_C11ptype_estimate}  and   Proposition \ref{lemma:Hessian_W2p_estimate} with   $p=n+1$.

\begin{prop}
    \label{prop: Hessian}
     Let $u$ be a smooth 2-convex solution of \eqref{eq:sigma_lambda} in $ B_2$ with $f\in C^{\infty}(B_2)$ and $\inf_{B_2} f\ge f_0 >0$.
     Let $w \in C^2(\mathbb{R}^n)$ be a 2-convex function, and let $\Omega$ be a connected component of $\{ w >u \}$ such that $\overline\Omega \subset B_1$.
     Define $$ \varphi(x) := (w(x)-u(x))^6, \qquad x\in \Omega. $$
     Then $\varphi\in C^2(\Omega)$ satisfying $\varphi=0$ on $\partial\Omega$, and
     $$ \| \varphi^{2n} \Delta u\|_{L^\infty(\Omega)} \leq C(n,f_0,\|w\|_{C^{0,1}(B_2)},\|u\|_{C^{0,1}(B_2)},\|f\|_{C^{0,1}(B_2)}). $$
\end{prop}

\section{Interior \texorpdfstring{$C^{2}$}{} regularity}
\label{section:C2-regularity}
In this section we prove Theorem \ref{Thm_sigma}. Assume that $u$ is a convex viscosity solution to
\begin{equation*}
    \sigma_{2}(D^2 u) = \sum_{1\leq i < j \leq n} \lambda_i \lambda_j = f(x), \quad x\in B_2,
\end{equation*}
with $f\in C^{0,1}(B_2)$, and $\inf_{B_2} f \geq f_0>0$.
The proof consists of two main parts:
\begin{enumerate}
    \item  Construct a sequence  of smooth approximating solutions $\{v_k\}$  such that $v_k \to u$  uniformly;
    \item  Prove the uniform $C^{1,1}$ estimates for $\{v_k\}$.
\end{enumerate}

\subsection{Approximating solutions and convergence}
    We begin by selecting smooth approximations of the data.
    Let $\{f_k\} \subset C^\infty(B_2)$ and $\{u_k\} \subset C^\infty(\overline{B_1})$ satisfy
    \begin{equation}
        \|f_k - f\|_{L^\infty(B_1)} \leq \frac{f_0}{2k}, \qquad \|f_k\|_{C^{0,1}(B_1)} \leq C \|f\|_{C^{0,1}(B_2)}, \qquad
        \|u_k - u\|_{C^0(\overline{B_1})} \leq \frac{1}{k},
        \label{eq:approx_norms}
    \end{equation}
    for each $k\geq 1$.
    For every $k$, we consider the Dirichlet problem
    \begin{equation}
    \label{eq:Dirichlet-approx}
        \begin{cases}
        \sigma_2(D^2 v_k) = f_k & \text{in}\ B_1, \\
        v_k = u_k & \text{on}\ \partial B_1.
        \end{cases}
    \end{equation}
    By the classical result of Caffarelli-Nirenberg-Spruck \cite{MR0806416}, problem \eqref{eq:Dirichlet-approx} admits a unique smooth 2-convex solution $v_k$.

    To obtain $C^0$ convergence of $v_k$ to $u$, we establish the following stability estimate. The argument is standard, see for instance Ishii-Lions \cite{Ishii_LionsViscositysolu}, Trudinger \cite{trudinger1990dirichlet} and Lu-Tsai \cite{LusiyuanPogorelovQuotient}.

\begin{prop}
\label{prop:stability}
    Let $u$ be a convex viscosity solution of $\sigma_2(D^2u) = f$ in $B_2$ and let $v\in C^{\infty}(\overline{B_1})$ be a smooth 2-convex solution of $\sigma_2(D^2 v) = g$ in $B_1$. Assume $f\in C^0(B_1)$, $f\geq f_0>0$ and $ \|f-g\|_{L^\infty(B_1)} \leq \frac{1}{2}f_0$. Then
    \begin{equation}
    \label{eq:stability}
        \|u-v\|_{C^0(\overline{B_1})} \leq \sup_{\partial B_1}|u-v|+\frac{1}{\sqrt{f_0}}\|f-g\|_{L^\infty(B_1)}.
    \end{equation}
\end{prop}

\begin{proof}
    Let $\epsilon:= f_0^{-\frac{1}{2}} \|f-g\|_{L^\infty(B_1)}$.
    If $\|f-g\|_{L^\infty(B_1)} = 0$, we simply take an arbitrary $\epsilon>0$, and let $\epsilon \to 0$ at the end.

    Fix \(x_0 \in B_1\) and suppose \(\varphi \in C^2(B_1)\) touches $ u + \epsilon|x|^2$ from above at $x_0$, i.e.
    $$
    \varphi(x_0) = u(x_0)+\epsilon|x_0|^2, \qquad  \varphi \geq u+\epsilon|x|^2 \  \text{ in }  B_1.
    $$
    Then $\varphi-\epsilon|x|^2$ touches $u$ from above at $x_0$.
    Since \(u\) is a viscosity subsolution of \(\sigma_2(D^2u) = f\), by the definition in A.1 of the Appendix we have
    $$
    \sigma_2\left(D^2 \varphi(x_0) -2\epsilon I_n\right) \geq f(x_0).
    $$
    Using the identity
    $$ \sigma_2(M + 2\epsilon I) = \sigma_2(M) + 2\epsilon(n-1)\operatorname{tr}(M) + 2n(n-1) \epsilon^2, $$
    and the fact that $\Delta(\varphi-\epsilon|x|^2)(x_0)>\sqrt {2f(x_0)}>0$,
    we obtain
    \begin{equation*}
        \sigma_2(D^2\varphi(x_0)) = \sigma_2(D^2(\varphi-\epsilon|x|^2)(x_0) + 2\epsilon I) > f(x_0) + 2\sqrt{2f_0}\epsilon(n-1).
    \end{equation*}
    So $u+\epsilon|x|^2$ is a viscosity subsolution to $\sigma_2 (D^2u)= f + 2\sqrt{2f_0}\epsilon(n-1)$.
    Define
    \begin{equation*}
     w(x) := u(x)  + \epsilon(|x|^2-1)-\sup_{\partial B_1}|u-v|.
    \end{equation*}
    Then $w$ is also a viscosity subsolution to $\sigma_2 = f + 2\sqrt{2f_0}\epsilon(n-1)$.
    Moreover, $w \leq v$ on $\partial B_1$.

   Next, we   claim that $w \leq v$ in $B_1$.
  Assume, to the contrary, that $\max_{B_1}(w-v) = c_0 >0$ is attained at some interior point $y_0$:
    \begin{eqnarray}
    \label{eq:w-c_0-touch-v}
        (w-c_0)(y_0)=v(y_0), \qquad   w-c_0 \leq v \ \text{ in } B_1 .
    \end{eqnarray}
    This implies that $v$ touches $w-c_0$ from above at $y_0$. Since $w$ is a subsolution, we get $$\sigma_2(D^2 v)(y_0) \geq f(y_0) + 2\sqrt{2f_0}\epsilon(n-1).$$
    But $\sigma_2(D^2v) = g \leq f + \|f-g\|_{L^\infty(B_1)}$, which gives a contradiction, because
    $$ 2\sqrt{2f_0}\epsilon(n-1) = 2\sqrt{2}(n-1) \| f-g \|_{L^\infty(B_1)} > \| f-g \|_{L^\infty(B_1)} $$
    for $n\geq 2$.
    Therefore we must have $w \le v$ in $B_1$, i.e.
    \begin{equation}
    \label{eq:upper_u-v}
        u-v \le \epsilon+\sup_{\partial B_1}|u-v|.
    \end{equation}
    Similarly, define
    \begin{eqnarray*}
         \widetilde{w} = v + \epsilon(|x|^2-1)-\sup_{\partial B_1}|u-v|.
    \end{eqnarray*}
    We claim that $\widetilde{w} \leq u$ in $B_1$.
    Assume the contrary that $\widetilde{w} - u$ attains a positive maximum at $\widetilde{y}_0\in B_1$ with
    \begin{eqnarray*}
    (\widetilde{w}-\widetilde{c}_0)(\widetilde{y}_0) = u(\widetilde{y}_0), \qquad \widetilde{w}-\widetilde{c}_0 \leq u \ \text{ in } B_1.
    \end{eqnarray*}
    This implies that $\widetilde{w}-\widetilde{c}_0$ touches $u$ from below at $\widetilde{y}_0$. Since $u$ is a viscosity supersolution to $\sigma_2 = f$, we have
    \begin{equation}
    \label{eq:tilde-w}
        \sigma_2(D^2\widetilde{w}(\widetilde{y}_0)) \le f(\widetilde{y}_0)\le g(\widetilde{y}_0) + \|f-g\|_{L^\infty(B_1)}.
    \end{equation}
    On the other hand, $\sigma_2(D^2 v)=g\ge f- \frac{1}{2}f_0\ge \frac{1}{2}f_0$, we have
    \begin{eqnarray*}
         \sigma_2(D^2\widetilde{w}) &=& \sigma_2(D^2v + 2\epsilon I) \\
         & = & \sigma_2(D^2 v) + 2\epsilon(n-1) \Delta v + 2n(n-1)\epsilon^2  \\
         & > & g+2\sqrt{f_0}\epsilon(n-1) \\
         & = & g + 2(n-1) \|f-g\|_{L^\infty(B_1)},
    \end{eqnarray*}
    which contradicts \eqref{eq:tilde-w}. Therefore $\widetilde{w}\le u $ in $B_1$, and thus
    \begin{equation}
    \label{eq:lower_u-v}
        u-v\ge -\epsilon-\sup_{\partial B_1}|u-v|.
    \end{equation}
    Combining \eqref{eq:upper_u-v}, \eqref{eq:lower_u-v} and the definition of $\epsilon$, we get the desired conclusion.
\end{proof}

Applying Proposition \ref{prop:stability} to $u$ and $v_k$ and using \eqref{eq:approx_norms}, we obtain    the following uniform convergence.

\begin{corollary}
\label{coro:uniform_v_k}
    For all $k \geq 1$,
    \begin{eqnarray*}
    \|v_k - u\|_{C^0(\overline{B_1})} \leq \|u_k - u\|_{C^0(\overline{B_1})} +  \frac{1}{\sqrt{f_0}} \|f_k - f\|_{L^\infty(B_1)} \leq \frac{C}{k},
\end{eqnarray*}
where $C$ only depends on $f_0$. Consequently, $v_k \to u$ uniformly in $B_1$ as $k\to \infty$.
\end{corollary}

\subsection{Strict 2-convexity and uniform \texorpdfstring{$C^{1,1}$}{} estimates}

To use Proposition \ref{prop: Hessian} to obtain a uniform $C^{1,1}$ estimate for $v_k$, we require a 2-convex barrier function $w$.
A key property is the strict 2-convexity of convex viscosity solution to \eqref{eq:sigma_lambda}.
This property was established by Mooney \cite{MR4246798}, and it ensures the existence of a 2-convex barrier function.
The following lemma, which we quote from \cite{MR4246798} for the reader's convenience, gives the precise statement.

\begin{lemma}
\label{lemma:Mooney}
    Let $u$ be a convex viscosity solution of \eqref{eq:sigma_lambda} in $B_1$ with $f\geq f_0>0$.
    Then there exist a 2-convex function $w\in C^\infty(\mathbb R^n)$, constants $\delta,r>0$, and a connected open set $\Omega\subset B_{3/4}$ containing the origin, all depending only on $n$, $f_0$ and $\|u\|_{L^\infty(B_1)}$, such that
    \begin{itemize}
        \item [1.] $w>u$ in $\Omega$ and $w=u$ on $\partial\Omega$;
        \item [2.] $B_r \subset\Omega$ and $w-u\geq \delta $ in $B_r$.
    \end{itemize}
\end{lemma}

\begin{proof}
    After subtracting the supporting linear function at the origin, we may assume
    $$ u(0)=0,   \quad u\geq 0 \text{ in } B_1. $$
    By the strict 2-convexity result \cite[Theorem 1.1]{MR4246798}, the set $\{ u=0 \}$ is contained in an $(n-2)$-dimensional subspace.
    A compactness argument (see \cite[Proposition 4.1]{MR4246798}) yields a constant $\delta_{0}=\delta_0(n,f_0,\|u\|_{L^{\infty}(B_1)})>0$, such that, after a rotation,
    $$ u(x) > \delta_0, \quad \forall x\in  B_{3/4} \cap\{x_1^2 + x_2^2 \geq (2n)^{-2}\}. $$
    Consider the 2-convex function
    $$ w_0(x) := \frac{\delta_0}{3n} \big[ 2(n-2)(x_1^2+x_2^2)-(x_3^2+\cdots+x_n^2) \big]. $$
    Then
    \begin{align*}
        w_0(x) \leq \frac{\delta_0}{2} < u(x), \qquad  & x\in \partial B_{3/4} \cap\{x_1^2 + x_2^2 \geq (2n)^{-2}\}; \\
        w_0(x) < 0 \leq u(x), \qquad  & x\in \partial B_{3/4} \cap\{x_1^2 + x_2^2 < (2n)^{-2}\}.
    \end{align*}
    Hence $w_0 < u$ on $\partial B_{3/4}$ and there exists $\delta_1 = C(n,\delta_0, \|u\|_{L^\infty(B_1)})>0 $ such that
    $$ w_0(x) - u(x) \leq -\delta_1, \quad \forall x\in \partial B_{3/4}. $$
    Set $w(x) = w_0(x) + \delta_1/2 $.
    Then $w<u$ on $\partial B_{3/4}$, and $(w-u)(0)=\delta_1/2 > 0$.

    Let $\Omega$ be the connected component of $\{u<w\}$ that contains the origin. Because $w<u$ on $\partial B_{3/4}$, we have $\Omega\subset B_{3/4}$.
    Since $w-u$ is continuous and positive at the origin, there exists $r>0$, depending only on $\delta_1$ and $\|u\|_{L^\infty(B_1)}$, such that  $B_r \subset\Omega$
    and $w-u\geq \delta_1/4=:\delta $ in $B_r $.  This completes the proof.
\end{proof}

\begin{proof}[Proof of Theorem \ref{Thm_sigma}]
From Corollary \ref{coro:uniform_v_k}, we have $v_k \to u$ uniformly in $B_1$.
The gradient estimate for $k$-Hessian equations with Lipschitz right-hand side (see Chou-Wang \cite[Theorem 3.2]{MR1835381}) gives
$$ \|v_k\|_{C^{0,1}(B_{3/4})}  \leq  C(n, f_0, \|u\|_{L^\infty(B_1)}, \|f_k\|_{C^{0,1}(B_1)}). $$
Rescaling Proposition \ref{prop: Hessian} to $B_{3/4}$ and using Lemma \ref{lemma:Mooney}, we get
\begin{equation}
    \|D^2 v_k\|_{L^\infty(B_{1/2})} \leq C(n,f_0, \|u\|_{L^\infty(B_1)}, \|f\|_{C^{0,1}(B_1)}).
    \label{eq:v_k_C11_bound}
\end{equation}
With this $C^{1,1}$ estimate, the equation becomes uniformly elliptic on the branch $\{\Delta v_k > 0\}$. The celebrated Evans-Krylov-Safonov theory gives the $C^{2,\alpha}$ estimates for any $\alpha\in(0,1)$,
$$ \|v_k\|_{C^{2,\alpha}(\overline{B_{1/4}})} \leq C(n, \alpha, f_0, \|u\|_{L^\infty(B_1)}, \| f\|_{C^{0,1}(B_1)}). $$
Up to a subsequence, we have $v_k\overset{C^2}{\longrightarrow} u$, and
$$ \|u\|_{C^{2}(\overline{B_{1/4}})} \leq C(n, f_0, \|u\|_{L^\infty(B_1)}, \| f\|_{C^{0,1}(B_1)}), $$
which completes the proof of Theorem \ref{Thm_sigma}.
\end{proof}

\appendix
\section{Appendix}
\begin{lemma}
\label{lemma:optim}
Consider the optimization problem of the quadratic
\[
 q^*=\min\{q(x): x \in \mathbb{R}^n  \    \text{s.t.}  \quad d^T x = c\}, \qquad q(x):=  x^T A x - (b^T x)^2 ,
\]
where \(A \in \mathbb{R}^{n \times n}\) is symmetric positive definite, \(b, d \in \mathbb{R}^n\) with \(d \neq 0\), and \(c \in \mathbb{R}\). Define
\[
\alpha = b^T A^{-1} b, \quad \beta = b^T A^{-1} d, \quad \gamma = d^T A^{-1} d , \quad \eta = 1 - \alpha + \frac{\beta^2}{\gamma}.
\]
Then $\gamma>0$ and the following hold:

\begin{enumerate}
    \item The optimization problem is bounded below if and only if either
    \begin{itemize}
        \item \(\eta > 0\), or
        \item \(\eta = 0\) and \(\beta c = 0\).
    \end{itemize}

    \item If \(\eta > 0\), the  minimum value is
        \[
        q^* = \frac{c^2(1-\alpha)}{\gamma(1-\alpha) + \beta^2};
        \]
     If \(\eta = 0\) and \(\beta c = 0\), the minimum value is $q^* = \dfrac{c^2}{\gamma}$.
\end{enumerate}
\end{lemma}

\begin{proof}
Since \(A\) is positive definite, let \(A^{1/2}\) be its positive definite square root. Define
\[
y = A^{1/2} x, \quad \tilde{b} = A^{-1/2} b, \quad \tilde{d} = A^{-1/2} d.
\]
Then the problem becomes
\[
\min \{ \bigl( y^T y - (\tilde{b}^T y)^2 \bigr): \quad  y \in \mathbb{R}^n \quad \text{s.t.} \quad \tilde{d}^T y = c\}
\]
with
\[
\alpha = \tilde{b}^T \tilde{b}, \quad \beta = \tilde{b}^T \tilde{d}, \quad \gamma = \tilde{d}^T \tilde{d}, \quad \eta = 1 - \alpha + \frac{\beta^2}{\gamma}.
\]
Thus, without loss of generality, we may assume \(A = I\) and  consider
\[
\min \{ q(x): \quad x \in \mathbb{R}^n  \quad \text{s.t.} \quad d^T x = c\},   \qquad  q(x)=   \|x\|^2 - (b^T x)^2 .
\]
In this new expression, we have \(\alpha = b^T b = \|b\|^2\), \(\beta = b^T d\), \(\gamma = d^T d = \|d\|^2>0\), and \(\eta = 1 - \alpha + \beta^2/\gamma\).
The feasible set is \(S = \{ x \in \mathbb{R}^n : d^T x = c \}\). Since \(d \neq 0\) and $\gamma>0$, we choose a  solution
\[
x_0 = \frac{c}{\gamma} d,
\]
satisfying \(d^T x_0 = c\).
Every \(x \in S \) can be written as \(x = x_0 + z\) for some \(z \in D_0 = \{ z \in \mathbb{R}^n : d^T z = 0 \}\). Substituting into the objective function, we get
\[
q(x) =  \bigl[ \|z\|^2 - (b^T z)^2 \bigr]  + 2\bigl[ x_0 - (b^T x_0) b \bigr]^T z +  \|x_0\|^2 - (b^T x_0)^2 .
\]
Define
\[
Q(z) := \|z\|^2 - (b^T z)^2, \quad g := x_0 - (b^T x_0) b, \quad  q_0 := \|x_0\|^2 - (b^T x_0)^2 ,
\]
then $g$ is a constant vector and $q_0$ is a constant.

So the problem is reduced  to
\[
\min_{z\in D_0} q(z) := \min_{z\in D_0}\bigl[Q(z) + 2 g^T z + q_0\bigr].
\]
Decompose \(b\) into components parallel and orthogonal to \(d\) respectively,
\[
b = b_d + b_\perp, \quad b_d = \frac{\beta}{\gamma} d, \quad b_\perp = b - b_d.
\]
Then \(d^T b_\perp = 0\), so \(b_\perp \in D_0\). Moreover, $\|b_\perp\|^2 = b_\perp^T b_\perp = \alpha - \dfrac{\beta^2}{\gamma}.$
\begin{figure}[htbp]
        \centering
        \includegraphics[width=0.8\linewidth]{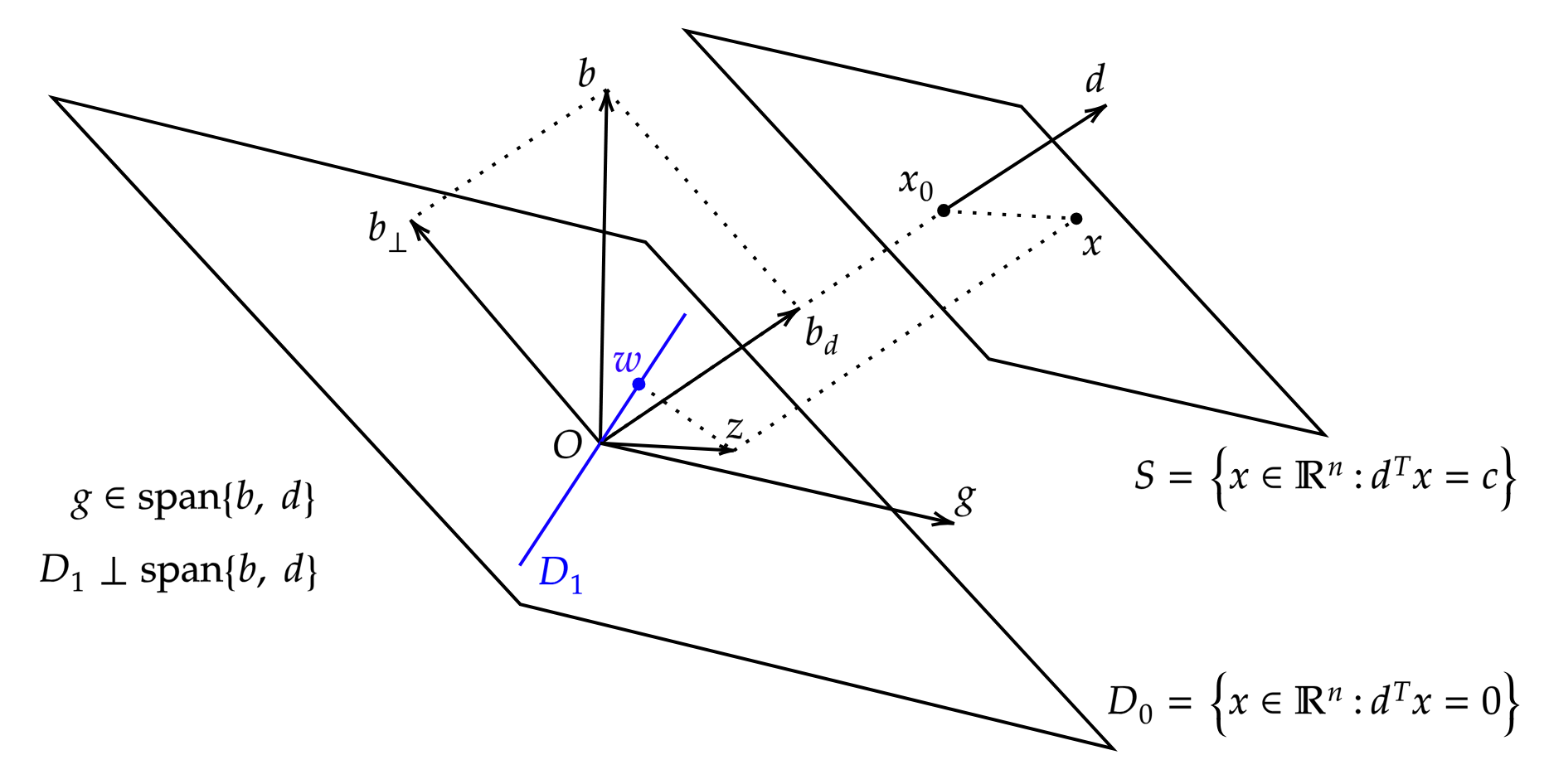}
        \caption{Decomposition of vectors}
        \label{fig:decompositions}
\end{figure}

Note that if \(b_\perp = 0\), then \(b\) is parallel to \(d\), and \(\eta = 1 - \alpha + \beta^2/\gamma =1 > 0\). So in the following we assume \(b_\perp \neq 0\) for the decomposition, since the case \(b_\perp = 0\) will be covered by the arguments of the case    \(\eta > 0\).

For any \(z \in D_0\), we have \(b^T z = b_\perp^T z\). Hence,
\[
Q(z) = \|z\|^2 - (b_\perp^T z)^2.
\]
Write \(z = t \frac{b_\perp}{\|b_\perp\|} + w\), where $t\in R$ and \(w \in D_1: = D_0 \cap \{x\in \mathbb{R}^n\mid b_\perp^T x = 0\}\). Then
\[
Q(z) = \|w\|^2 + t^2 (1 - \|b_\perp\|^2).
\]
Since \(\|b_\perp\|^2 = \alpha - \beta^2/\gamma\), we have
\begin{eqnarray}\label{eq:b_orthogonal}
    1 - \|b_\perp\|^2 = 1 - \alpha + \frac{\beta^2}{\gamma} = \eta.
\end{eqnarray}

\begin{itemize}
    \item \textbf{Case 1: \(\eta > 0\).} Then \(Q\) is positive definite on \(D_0\), so \(q(z)\) is strictly convex and is bounded below.
    Since $g = \frac{c}{\gamma}d - \frac{c\beta }{\gamma}b \in \operatorname{span}\{b,d\}$, we have $g \perp D_1$ and thus $g^Tw=0$. Then
    \begin{eqnarray*}
        q(z) &=& \|w\|^2 + \eta t^2 + 2g^T(w + t\frac{b_\perp}{\|b_\perp\|}) + q_0 \\
        &=& \|w\|^2 + \eta t^2 + \frac{2 g^T b_\perp}{\|b_\perp\|} t + q_0 \\
        & =& \|w \|^2 + \eta \left( t + \frac{g^T b_\perp}{\eta \|b_\perp\|} \right)^2 - \frac{(g^T b_\perp)^2}{\eta \|b_\perp\|^2} +q_0\\
        &:=& \|w \|^2 +f(t).
    \end{eqnarray*}
   Obviously, $\|w \|^2$ attains its  minimum  over \(w \in D_1\) at \(w^* = 0\), and $f(t) $ attains its  minimum  over  $\mathbb{R}$  at  \(t^* = -g^T b_\perp/(\eta \|b_\perp\|)\).
   The corresponding $q(x)$  attains its  minimum at \(x^* = x_0 + z^* = x_0 + w^* + t^* \frac{b_\perp}{\|b_\perp\|}\), which  can be expressed as
    \[
    x^* = \frac{c\beta}{\beta^2 + \gamma(1-\alpha)} b + \frac{c(1-\alpha)}{\beta^2 + \gamma(1-\alpha)} d.
    \]
    The minimum value is
    \[
    q^* = \frac{c^2(1-\alpha)}{\beta^2 + \gamma(1-\alpha)}.
    \]

    \item \textbf{Case 2: \(\eta = 0\).} In this case we have  $\|b_\perp\|^2=1.$  Write \(z = t b_\perp + w\) with \( w \in D_1\). Using $g\in span\{b,d\}\subset D_1^\perp$,
    we have
    \[
    q(z) = \|w\|^2  + 2t (g^T b_\perp) + q_0.
    \]
   Hence $q(z)$ is  bounded from below if and only if \(g^T b_\perp = 0\). Otherwise,
     \(q\to -\infty\)  when either $t\to +\infty $ or $t\to -\infty $.   Since
    \[
    g^T b_\perp = \frac{c}{\gamma} d^T b_\perp - \frac{c\beta}{\gamma} b^T b_\perp= - \frac{c\beta}{\gamma},
     \]
    so  \(g^T b_\perp = 0\) is equivalent  to \(\beta c = 0\).  In that case,
      \[
        q(z)=\|w\|^2+q_0.
        \]
  Therefore,  the    minimum  is attained at $w=0$ with arbitrary $t$, and the minimum value is
        \[
        q^* = q_0= \frac{c^2}{\gamma}.
        \]
\end{itemize}

Combining the above cases, the problem is bounded from below if and only if either \(\eta > 0\), or \(\eta = 0\) and \(\beta c = 0\). The
desired  minimum values  have been obtained. This completes the proof.
\end{proof}

\subsection{Viscosity solution of \texorpdfstring{$k$}{}-Hessian equation and Monge-Amp\`ere equation}

In this subsection, we show that a viscosity solution of a Monge-Amp\`ere equation in $\mathbb{R}^k$, when trivially extended to $\mathbb{R}^n$ with $2\leq k < n$, yields a viscosity solution of a corresponding $k$-Hessian equation.

\begin{definition}[Viscosity solution of Monge--Amp\`ere equation]
Let $\Omega \subset \mathbb{R}^k$,   $f \in C(\Omega)$ be a positive function, and $u \in C(\Omega)$ be a convex function.
\begin{enumerate}
    \item[(i)] We say $u$ is a viscosity subsolution of
    \begin{equation}
    \label{eq:appendix_MA}
        \det D^2 u = f \quad \text{ in } \Omega,
    \end{equation}
    if for any $x_0 \in \Omega$ and any $\psi \in C^2(\Omega)$ such that $D^2\psi(x_0) \geq 0$ and $u-\psi$ has a local maximum at $x_0$, then we have
    \[
    \det D^2\psi(x_0) \geq f(x_0).
    \]
    \item[(ii)] We say $u$ is a viscosity supersolution of \eqref{eq:appendix_MA} if for any $x_0 \in \Omega$ and any $\psi \in C^2(\Omega)$ such that $D^2\psi(x_0) \geq 0$ and $u-\psi$ has a local minimum at $x_0$, then we have
    \[
    \det D^2\psi(x_0) \leq f(x_0).
    \]
    \item [(iii)] We say $u$ is a viscosity solution of \eqref{eq:appendix_MA} if $u$ is both a viscosity subsolution and a viscosity supersolution of \eqref{eq:appendix_MA}.
\end{enumerate}
\end{definition}
It is well known that a viscosity solution is equivalent to an Aleksandrov solution if $f\in C(\Omega)$.

Let $\mathcal{S}_n$ denote the set of all $n\times n$ symmetric  matrices. For   $M \in  \mathcal{S}_n$, denote by $\sigma_j(M)$ the $j$-th elementary symmetric polynomial of its eigenvalues.
Let $\Gamma_k$ be the G\r{a}rding cone:
\[
\Gamma_k = \{ M \in \mathcal{S}_n : \sigma_j(M) > 0 \text{ for } j=1,\cdots,k \}.
\]

\begin{definition}[Viscosity solution for the $k$-Hessian equation]
Let $\Omega \subset \mathbb{R}^n$,   $F \in C(\Omega)$ be a positive function, and $v \in C(\Omega)$.
\begin{enumerate}
    \item[(i)] We say $v$ is a viscosity subsolution of
    \begin{equation}
    \label{eq:appendix_k-Hessian}
        \sigma_k(D^2 v) = F \quad \text{ in } \Omega,
    \end{equation}
    if for any $x_0 \in \Omega$ and any $\phi \in C^2(\Omega)$ such that $D^2 \phi(x_0) \in \Gamma_k$ and $v-\phi$ has a local maximum at $x_0$, then we have
    \[
    \sigma_k(D^2\phi)(x_0) \geq F(x_0).
    \]
    \item[(ii)] We say $v$ is a viscosity supersolution of \eqref{eq:appendix_k-Hessian} if for any $x_0 \in \Omega$ and any $\phi \in C^2(\Omega)$ such that $D^2 \phi(x_0) \in \Gamma_k$ and $v-\phi$ has a local minimum at $x_0$, then we have
    \[
    \sigma_k(D^2\phi)(x_0) \leq F(x_0).
    \]
    \item [(iii)] We say $v$ is a viscosity solution of \eqref{eq:appendix_k-Hessian} if $v$ is both a viscosity subsolution and a viscosity supersolution of \eqref{eq:appendix_k-Hessian}.
\end{enumerate}
\end{definition}

\begin{theorem}
\label{thm:appendix-viscosity}
Let $n > k \ge 1$.
Write $x = (x', x'') \in \mathbb{R}^n$ with $x' \in \mathbb{R}^k$, $x'' \in \mathbb{R}^{n-k}$.
Let $\Omega' \subset \mathbb{R}^k$ be a convex domain, and set $\Omega = \Omega' \times \mathbb{R}^{n-k} \subset \mathbb{R}^n$.
Assume $u \in C(\Omega')$ is a viscosity solution of
\[
\det D_{x'}^2 u = f(x') \qquad x' \in \Omega',
\]
where $f \in C(\Omega')$, $f>0$.
Define the trivial extension $v: \Omega \to \mathbb{R}$ by
\[
v(x) = v(x', x'') := u(x').
\]
Then $v$ is a viscosity solution of the $k$-Hessian equation
\begin{equation}
\label{eq:appendix_hessian-v}
\sigma_k(D_x^2 v) = F(x) \qquad x \in \Omega
\end{equation}
with $F(x) := f(x')$.
\end{theorem}

\begin{proof}
We verify that $v$ is both viscosity subsolution and supersolution of \eqref{eq:appendix_hessian-v}.

\noindent \textbf{Step 1: $v$ is subsolution of \eqref{eq:appendix_hessian-v}.}

Assume that  $\phi \in C^2(\Omega)$ and $x_0 = (x_0', x_0'') \in \Omega$ such that $v - \phi$ has a local maximum at $x_0$ and $D^2\phi(x_0) \in \Gamma_k$. We will show that $\sigma_k(D^2\phi)(x_0) \geq f(x_0')$.

Since $u$ is convex in $\mathbb{R}^k$, we have $v$ is also convex in $\mathbb{R}^n$. Then
$$ v(x) -v(x_0) \geq p \cdot (x-x_0), \quad \forall p \in \partial v(x_0). $$
Since $v-\phi$ has local maximum at $x_0$, we have
$$ \phi(x) -\phi(x_0) \geq v(x) - v(x_0) \geq p \cdot (x-x_0) $$
for all $x$ in a neighborhood of $x_0$.   This, together with   $\phi\in C^2$, implies that
$$ \phi(x) - \phi(x_0) = D\phi(x_0) \cdot(x-x_0) + \frac{1}{2} (x-x_0)^TD^2\phi(x_0) (x-x_0) + o(|x-x_0|^2) \geq p\cdot(x-x_0). $$
If $\partial v(x_0)$ has more than one element, this inequality cannot hold for all $p\in \partial v(x_0)$. Hence $\partial v(x_0) = \{ Dv(x_0)\}$ and $D\phi(x_0) = Dv(x_0)$.
Consequently,
$$ \frac{1}{2} (x-x_0)^TD^2\phi(x_0) (x-x_0) + o(|x-x_0|^2) \geq 0 $$
for all $x$ in a neighborhood of $x_0$, which forces $D^2 \phi(x_0) \geq 0$.
Denote
\[
M: = D^2\phi(x_0) = \begin{pmatrix} A & B \\ B^T & C \end{pmatrix} \geq 0,
\]
where $A := D_{x'}^2\phi(x_0) \in \mathbb{R}^{k \times k}, B:= D_{x'}D_{x''}\phi(x_0)\in \mathbb{R}^{k \times (n-k)}$, and $ C = D_{x''}^2\phi(x_0) \in \mathbb{R}^{(n-k) \times (n-k)}$.
Define $\psi(x') := \phi(x', x_0'')  $.
Then $\psi(x') \in C^2(\Omega')$ and  $D_{x'}^2\psi(x_0') = A$.
Since $v(x',x_0'') = u(x')$, we have
\[
\psi(x') - u(x') = \phi(x',x_0'') - v(x',x_0'').
\]
Because $v-\phi$ attains a local maximum at $(x_0',x_0'')$, the function $u - \psi$ attains a local maximum at $x_0'$.
Since $u$ is a viscosity solution of $\det D_{x'}^2 u = f$, the subsolution condition gives
\[
\det A = \det D_{x'}^2\psi(x_0') \ge f(x_0').
\]
Since $M$ is positive semidefinite, then
\[
\sigma_k(M) = \sum_{1 \le i_1 < \dots < i_k \le n} \det M[i_1,\dots,i_k] \geq \det M[1,\cdots,k] = \det A,
\]
where $M[i_1,\dots,i_k]$ is a $k \times k$ principal minor of $M$, and the  inequality holds since the principal minors of a positive semidefinite matrix are also positive semidefinite.
Hence
\[
\sigma_k(M) \ge \det A \ge f(x_0'),
\]
which yields the subsolution property for $v$.

\medskip
\noindent \textbf{Step 2: $v$ is supersolution of \eqref{eq:appendix_hessian-v}.}

Assume that $\phi \in C^2(\Omega)$ and $x_0 = (x_0', x_0'') \in \Omega$ such that $v - \phi$ has a local minimum at $x_0$ and $D^2\phi(x_0) \in \Gamma_k$.
We will show that $\sigma_k(D^2\phi(x_0)) \le f(x_0')$.

Again write
$$ M:= D^2\phi(x_0) = \begin{pmatrix} A & B \\ B^T & C \end{pmatrix}. $$
First, we claim $C \le 0$.
Indeed, fix any $\eta'' \in \mathbb{R}^{n-k}$ and consider
\[
h(s) = \phi(x_0', x_0'' + s\eta'') - v(x_0', x_0'' + s\eta'') = \phi(x_0', x_0'' + s\eta'') - u(x_0').
\]
Since $v-\phi$ has a local minimum at $x_0$, $h(s)$ has a local maximum at $s=0$.
Thus $$h''(0) = \langle C\eta'', \eta'' \rangle \le 0. $$
Because $\eta''$ is arbitrary, we have $C \leq 0$.
Since $M\in \Gamma_k$, we know $M$ has at least $k$ positive eigenvalues. So $C \leq 0$ forces $A $ is positive semidefinite.

Define $\psi(x') := \phi(x', x_0'')$, then $D_{x'}^2\psi(x_0') = A>0$.
Moreover,
\[
u(x') - \psi(x') = v(x',x_0'') - \phi(x',x_0'') \ge v(x_0',x_0'') - \phi(x_0',x_0'') = u(x_0') - \psi(x_0')
\]
for $x'$ near $x_0'$. So $u-\psi$ has a local minimum at $x_0'$.
The supersolution condition of $u$ gives
\[
\det A = \det D_{x'}^2 \psi(x_0') \le f(x_0').
\]
It remains to prove that $\sigma_k(M) \leq \det A$. We first claim the following inequality for any $X, Y \in \Gamma_k$, that is
\begin{equation}
\label{eq:appendix-X-Y}
    k \sigma_k^{1/k}(Y) \sigma_k(X)^{1-1/k} \leq \left. \frac{d}{dt} \right|_{t=0} \sigma_k(X+tY).
\end{equation}
In fact, the function $X \mapsto \sigma_k(X)^{1/k}$ is concave on $\Gamma_k$, so for any $X, Y \in \Gamma_k$,
\begin{equation}
\label{eq:appendix-concave}
\sigma_k(Y)^{1/k} \le \sigma_k(X)^{1/k} + \frac{1}{k} \sigma_k(X)^{1/k - 1} \sum_{i,j=1}^n \frac{\partial \sigma_k(X)}{\partial X_{ij}} (Y_{ij} - X_{ij}).
\end{equation}
Since
$$ \sum_{i,j=1}^n \frac{\partial \sigma_k(X)}{\partial X_{ij}} Y_{ij} = \left. \frac{d}{dt} \right|_{t=0} \sigma_k(X+tY), $$
and
$$ \sum_{i,j=1}^n \frac{\partial \sigma_k(X)}{\partial X_{ij}} X_{ij} = k \sigma_k(X),$$
plugging into \eqref{eq:appendix-concave} yields \eqref{eq:appendix-X-Y}.

Now we choose
\[
X = \begin{pmatrix} A & 0 \\ 0 & 0 \end{pmatrix}, \qquad Y = M = \begin{pmatrix} A & B \\ B^T & C \end{pmatrix}.
\]
Then $\sigma_k(X) = \det A$, and
\[
X + tY = \begin{pmatrix} (1+t)A & tB \\ tB^T & tC \end{pmatrix}.
\]
Expand $\sigma_k$ of this matrix as a sum of $k \times k$ principal minors.
Only terms that are linear in $t$ contribute to the derivative at $t=0$.
There are three cases for the choice of the $k$ indices:
\begin{itemize}
\item The minor with indices $1,\cdots,k$ gives
$$\det((1+t)A) = (1+t)^k \det A = \det A + k \det A \cdot t + O(t^2). $$
\item Minors containing exactly one index from $\{k+1,\dots,n\}$ contribute
$$ \operatorname{tr}(C) \sigma_{k-1}(A) \cdot t + O(t^2). $$
\item Minors with at least two indices from $\{k+1,\cdots,n\}$ give $O(t^2)$.
\end{itemize}
Hence
\[
\left. \frac{d}{dt} \right|_{t=0} \sigma_k  \begin{pmatrix} (1+t)A & tB \\ tB^T & tC \end{pmatrix} = k \det A + \operatorname{tr}(C) \sigma_{k-1}(A) \leq k \det A,
\]
 where we have used the fact   $C \le 0$ and $\sigma_{k-1}(A) > 0$.
Substituting this estimate into \eqref{eq:appendix-X-Y} yields
\[
 k \sigma_k(M)^{1/k} (\det A)^{1-1/k} \leq k \det A,
\]
which is exactly $\sigma_k(M) \leq \det A \leq f(x_0')$.

Thus $v$ satisfies both the subsolution and supersolution conditions for $\sigma_k(D^2 v) = F$, and therefore $v$ is a viscosity solution of the $k$-Hessian equation.
\end{proof}

\printauthorinfos

\end{document}